\documentclass[12pt]{article}
\usepackage{caption,cite}
\usepackage{amsthm,amsmath,amssymb,mathrsfs}
\usepackage{graphicx}
\usepackage{tikz}
\usepackage[colorlinks=true,citecolor=black,linkcolor=black,urlcolor=blue]{hyperref}

\newtheorem{thm}{Theorem}

\newtheorem{theorem}{Theorem}[section]
\newtheorem{lemma}[theorem]{Lemma}
\newtheorem{cor}[theorem]{Corollary}
\newtheorem{remark}[theorem]{Remark}

\theoremstyle{definition}

\theoremstyle{remark}

\DeclareMathOperator{\cp}{cp}

\DeclareMathOperator{\scp}{scp}
\newcommand\kk{k}

\title{\bf Pairwise Balanced Designs \\ and \\ Sigma Clique Partitions}
\author{Akbar Davoodi \qquad Ramin Javadi \qquad Behnaz Omoomi\\[4pt]
\small Department of Mathematical Sciences\\[-0.8ex]
\small Isfahan University of Technology\\[-0.8ex] 
\small 84156-83111, Isfahan, Iran\\
}

\date{}

\begin{document}

\maketitle

\begin{abstract}
In this paper, we are interested in minimizing the  sum of block sizes in a pairwise balanced design, where there are some constraints on the size of one block or the size of the largest block. For every positive integers $n,m$, where $m\leq n$, let $\mathscr{S}(n,m)$ be the smallest integer $s$ for which there exists a PBD on $n$ points whose largest block has size $m$ and  the sum of its block sizes is equal to $s$. Also, let  $\mathscr{S}'(n,m)$ be the smallest integer $s$ for which there exists a PBD on $n$ points which has a block of size $m$ and the  sum of it block sizes is equal to $s$. 
We prove some lower bounds for $\mathscr{S}(n,m)$ and $\mathscr{S}'(n,m)$. Moreover,   we apply these bounds to determine the asymptotic behaviour of the sigma clique partition number of the graph $K_n-K_m$, Cocktail party graphs and complement of paths and cycles.

\end{abstract}
\section{Introduction}
An $(n,k,\lambda)-$\textit{design} (or $(n,k,\lambda)$-BIBD) is a pair $(P, \mathcal{B})$ where $P$ is a  finite set of $n$ \textit{points} and $\mathcal{B}$ is a collection of $k-$subsets of $P$, called  \textit {blocks},  such that every two distinct points in $P$  is contained in exactly $\lambda$ blocks. 
In case $|P|={|\cal B|}$, it is called a \textit{symmetric design}. For positive integer $ q $,  
a $(q^2+q+1, q+1, 1)$-BIBD and a $(q^2,q,1)$-BIBD are called a \textit{projective plane} and an \textit{affine plane} of order $q$, respectively.
A design is called \textit{resolvable}, if there exists a  partition  of the set of blocks
$\mathcal{B}$ into \textit{parallel classes}, each of which is a partition of $P$.

A \textit{pairwise balanced design} (PBD) is a pair $(P,\mathcal{B})$, where $P$ is a finite set of $n$ points and $\mathcal{B}$ is  a family of subsets of $P$, called \textit {blocks}, such that every two distinct points in $P$, appear in exactly one block. A \textit{nontrivial} PBD is a PBD where $P\not \in \mathcal{B}$. A PBD $(P,\mathcal{B})$ on $n$ points with one block of size $n-1$ and the others of size two is called \textit{near-pencil}.

The problem of determining the minimum number of blocks in a pairwise balanced design when the size of its largest block is specified or the size of a particular block is specified, has been the subject of many researches in recent decades.
The most important and well-known result about this problem is due to de Bruijn and Erd\H{o}s~\cite{deBruijn48} which states that every nontrivial PBD on $n$ points has at least $n$ blocks and the only nontrivial PBDs on $n$ points with exactly $n$ blocks  are near-pencil and projective plane.
For every positive integers $n,m$, where $m\leq n$, let $\mathscr{G}(n,m)$ be the minimum number of blocks in a PBD on $n$ points whose largest block has size $m$. Also, let  $\mathscr{G}'(n,m)$ be the minimum number of blocks in a PBD on $n$ points which has a block of size $m$.
A classical result known as Stanton-Kalbfleisch Bound \cite{stanton70} states that $\mathscr{G}'(n,m)\geq 1+(m^2(n-m))/(n-1)$ and equality holds if and only if there exists a resolvable $(n-m,(n-1)/m,1)$- BIBD. Also, a corollary of Stanton-Kalbfleisch is that $\mathscr{G}(n,m)\geq \max\{n(n-1)/m(m-1), 1+(m^2(n-m))/(n-1)\}$. For a survey on these and more bounds, see \cite{Rees90,stanton97}.

In this paper, we are interested in minimizing the  sum of block sizes in a PBD, where there are some constraints on the size of one block or the size of the largest block. For every positive integers $n,m$, where $m\leq n$, let $\mathscr{S}(n,m)$ be the smallest integer $s$ for which there exists a PBD on $n$ points whose largest block has size $m$ and  the sum of its block sizes is equal to $s$. Also, let  $\mathscr{S}'(n,m)$ be the smallest integer $s$ for which there exists a PBD on $n$ points which has a block of size $m$ and the  sum of it block sizes is equal to $s$. 
 In Section~\ref{sec:pbd}, we prove some lower bounds for $\mathscr{S}(n,m)$ and $\mathscr{S}'(n,m)$. In particular, we show that $\mathscr{S}(n,m)\geq 3n-3$, for every $m$, $2\leq m\leq n-1$. Also, we prove that, for every $2\leq m\leq n $,
 \[\mathscr{S}'(n,m)\geq \max\left\{(n+1)m-\frac{m^2(m-1)}{n-1}, m+\frac{(n-m)(n-5m-1)}{2}\right\},\]
 where equality holds for $ m\geq n/2 $. Furthermore, we prove that if $ n\geq 10 $ and $2\leq m\leq n-\frac{1}{2} ( \sqrt{n} +1)$, then
 $\mathscr{S}(n,m)\geq n(\lfloor\sqrt{n}\rfloor+1)-1$.

The connection of pairwise balanced designs and clique partition of graphs is already known in the literature. Given a simple graph $ G $, by a \textit{clique} in $G$ we mean a subset of mutually adjacent vertices. A \textit{clique partition } $\mathcal{C}$ of $G$ is a family of cliques in $G$ such that the endpoints of every edge of $G$ lie in exactly one member  of $\mathcal{C}$. The minimum size of a clique partition of $G$ is called the  \textit{clique partition number} of $G$ and is denoted by $\cp(G)$.  

For every graph $G$ with $ n $ vertices, the union of a clique partition of $G$ and a clique partition of its complement, $\overline{G}$, form a PBD on $n$ points. This connection has been deployed to estimate $\cp(G)$, when $G$ is some special graph such as $K_n-K_m$~\cite{WallisAsymp,erdos88,pullman82,Rees}, Cocktail party graphs and complement of paths and cycles \cite{Wallis82,Wallis84,Wallis87}. 

Our motivation for study of the above mentioned problem is a weighted version of clique partition number. The \textit{sigma clique partition number} of a graph $G$, denoted by $ \scp(G) $, is defined as the smallest integer $s$ for which there exists a clique partition of $G$ where  the sum of  the sizes of its cliques is equal to $s$. It is shown that for every graph $G$ on $n$ vertices, $\scp(G)\leq \lfloor n^2/2\rfloor$, in which equality holds if and only if $ G $ is the complete bipartite graph $K_{\lfloor n/2\rfloor, \lceil n/2\rceil } $  \cite{chung81,kahn81,gyori80}.

Given a clique partition $\mathcal{C}$  of a graph $G$, for every vertex $x\in V(G)$, the \textit{valency} of $x$ (with respect to $\mathcal{C}$), denoted by $v_\mathcal{C}(x)$, is defined to be the number of cliques in $\mathcal{C}$ containing $x$. In fact,
\[\scp(G)=\min_{\mathcal{C}} \sum_{C\in \mathcal{C}}|C|= \min_{\mathcal{C}} \sum_{x\in V(G)} v_{\mathcal{C}}(x),\]
where the minimum is taken over all possible clique partitions of $G$.

 In Section~\ref{sec:Kn-Km}, we apply the results of Section~\ref{sec:pbd} to determine the asymptotic behaviour of the sigma clique partition number of the graph $K_n-K_m$. In fact, we prove that if $m\leq {\sqrt{n}}/{2}$, then $\scp(K_n-K_m)\sim (2m-1)n$, if ${\sqrt{n}}/{2}\leq m\leq \sqrt{n}$, then $\scp(K_n-K_m)\sim n\sqrt{n}$ and if $m\geq \sqrt{n}$ and $m=o(n)$, then $\scp(K_n-K_m)\sim mn$. Also, if $G$ is Cocktail party graph, complement of path or cycle on $n$ vertices, then we prove that $\scp(G)\sim n\sqrt{n}$.
%

\section{Pairwise balanced designs}\label{sec:pbd}

A celebrated result of de Bruijn and Erd\H{o}s states that for every nontrivial PBD  $(P,\mathcal{B})$, we have $|\mathcal{B}|\geq |P|$ and equality holds if and only if  $(P,\mathcal{B})$ is near-pencil or projective plane~\cite{deBruijn48}. In this section, we are going to answer the question that what is the minimum sum of block sizes in a PBD.

 The following theorem can be viewed as a de Bruijn-Erd\H{o}s-type bound, which shows that $\mathscr{S}(n,m)\geq 3n-3$, for every $m$, $2\leq m\leq n-1$.

\begin{theorem}\label{thm:PBDsigma}
Let $(P,\mathcal{B})$ be a nontrivial PBD with $n$ points, then we have
\begin{equation}\label{eq:pbd}
\sum_{B\in\mathcal{B}} |B| \geq 3n-3,
\end{equation}
and equality holds if and only if $(P,\mathcal{B})$ is near-pencil.
\end{theorem}
\begin{proof}
We use induction on the number of points. Let $(P,\mathcal{B})$ be a nontrivial PBD with $n$ points. Inequality~\eqref{eq:pbd} clearly holds when $n=3$. So assume that $n\geq 4$ and for every $x\in P$, let $r_x$ be the number of blocks containing $x$. First note that for every block $B\in\mathcal{B}$ and every $x\in P\setminus B$, we have $r_x\geq |B|$.  

If there is a block $B_0\in \mathcal{B}$ of size $n-1$ and $x_0$ is the unique point in 
$P\setminus B_0$, then for every $x\in B_0$, $x$ and $x_0$ appear within a block of size two. Therefore,  $(P,\mathcal{B})$ is near-pencil and $\sum_{B\in\mathcal{B}}|B|=(n-1)+2(n-1)=3n-3$. 

Otherwise,  all blocks are of size at most $n-2$. First we prove that there exists some point $x\in P$ with $r_x\geq 3$. Since there is no block of size $n$, $r_x\geq 2$ for all $x\in P$. Now for some $y\in P$, assume that $B_1, B_2$ are the only two blocks containing $y$. Since $n\geq 4$, the size of at least one of these blocks, say $B_1$, is greater than two. Let $x\neq y$ be an element of $B_2$. Then,  $r_x\geq |B_1|\geq 3$. 
Hence, there exists some point $x\in P$ which appears in at least three blocks.

Now, remove $x$ from all blocks to obtain the nontrivial PBD $(P',\mathcal{B'})$, where $P'=P\setminus \{x\}$ and $\mathcal{B'}=\{B\setminus\{x\}\ :\  B\in\mathcal{B}\}$. Therefore, 
\begin{equation}\label{eq:case2}
\sum_{B\in \mathcal{B}}|B| = r_x+\sum_{B'\in\mathcal{B'}}|B'|\geq 3+3(n-2),
\end{equation}
where the last inequality follows from the induction hypothesis.

Now, assume that for a  PBD $(P,\mathcal{B})$ equality holds in \eqref{eq:pbd}. If $(P,\mathcal{B})$ is not a near-pencil, then equality holds in \eqref{eq:case2} as well and thus  we have $2\leq r_x\leq 3$, for every $x\in P$. On the other hand, $\sum_{B\in\mathcal{B}}|B|=\sum_{x\in P} r_x=3n-3$. Therefore, there are exactly $3$ points, say $x,y,z$, each of which appears in exactly two blocks and each of the other points appears in exactly three blocks. Also, let $B_1, B_2$ be the only two blocks containing $y$ and assume that $x\in B_1$. Therefore, $2=r_x\geq |B_2|$ and then $|B_1|=n-1$, which is a contradiction.
\end{proof}
Since  the union of every clique partition of $G$ and $\overline{G}$ forms a clique partition for $K_n$ which is equivalent to a PBD on $n$ points, the following corollaries are straightforward.
\begin{cor}
Let $\mathcal{C}$ be a clique partition of $K_n$ whose cliques are of size at most $n-1$. Then,  $\sum_{C\in \mathcal{C}} |C|\geq 3n-3$.
\end{cor}

\begin{cor}
For every graph $G$ on $n$ vertices except the empty and  complete graph, we have 
\[\scp(G)+\scp(\overline{G})\geq 3n-3,\]
and equality holds if and only if $G$ or $\overline{G}$ contains a clique of size $n-1$.
\end{cor}

In the same vein, one can prove the following theorem which states a lower bound on the maximum number of appearance of the points in a PBD.

\begin{theorem}
Let $(P,\mathcal{B})$ be a nontrivial PBD with $n$ points, and for every $x\in P$, let $r_x$ be the number of blocks containing $x$. Then,  we have
\begin{equation}\label{eq:pbd2}
\max_{x\in P}{r_x}\geq \frac{1+\sqrt{4n-3}}{2},
\end{equation}
and equality holds if and only if $(P,\mathcal{B})$ is a projective plane or near-pencil.
\end{theorem}
\begin{proof}
Let $(P,\mathcal{B})$ be a nontrivial PBD with $n$ points and define $r=\max_{x\in P} r_x$. Fix a point $x\in P$ and let $\mathcal{B}_x \subset \mathcal{B}$ be the set of blocks containing $x$. The family of sets $\{B\setminus \{x\} \ : \  B\in \mathcal{B}_x\}$ is a partition of the set $P\setminus \{x\}$. Thus, 
\begin{equation}\label{eq:PBDproof1}
n-1= \sum_{B\in\mathcal{B}_x}(|B|-1)\leq r_x(\max_{B\in\mathcal{B}_x} |B|-1).
\end{equation}
Therefore,  there exists some block $B_0$ containing $x$, where $r_x (|B_0|-1)\geq n-1$. Now, let $y$ be a point not in $B_0$. By a note within the proof of Theorem~\ref{thm:PBDsigma}, we have $r_y\geq |B_0|$ and then
\begin{equation}\label{eq:PBDproof2}
r(r-1)\geq r_x(r_y-1)\geq r_x(|B_0|-1)\geq n-1.
\end{equation}
This yields the assertion.

 Now, assume that  equality holds in $\eqref{eq:pbd2}$. Then,  we have  equalities in \eqref{eq:PBDproof1} and \eqref{eq:PBDproof2}. Thus, all valencies $r_x$ are equal and all blocks have the same size, say $k$, which shows that $(P,\mathcal{B})$ is an $(n,k,1)-$design. Also by \eqref{eq:PBDproof2}, we have $r=k$, i.e. $(P,\mathcal{B})$ is a symmetric design.
\end{proof}

Although the given bound in \eqref{eq:pbd} is sharp, it can be improved if the PBD avoids blocks of large sizes. The following theorem, as an improvement of Theorem~\ref{thm:PBDsigma}, provides some lower bounds on the sum of block sizes, when there are some constraints on the size of a block.

\begin{theorem} \label{thm:PBD}
If  $(P,\mathcal{B})$ is  a PBD with $n$ points where $\tau$ is the maximum size of blocks in $\mathcal{B}$, then
\begin{equation}\label{eq:pbdA}
\sum_{B\in \mathcal{B}} |B| \geq \frac{n(n-1)}{\tau-1}.
\end{equation}
Also if there is a block of size $\kk$, then 
\begin{equation}\label{eq:pbdB}
\sum_{B\in \mathcal{B}} |B| \geq (n+1) \kk - \frac{\kk^2(\kk-1)}{n-1},
\end{equation}
and
\begin{equation}\label{eq:pbdC}
\sum_{B\in \mathcal{B}} |B| \geq \kk-\frac{(n-\kk)(n-5\kk-1)}{2}.
\end{equation}
Moreover, if $\kk\geq n/2$, then there exists a PBD on $n$ points with a block of size $\kk$, for which equality holds in {\em (\ref{eq:pbdC})}.
\end{theorem}

\begin{proof}
For every $x\in P$, let $r_x$ be the number of blocks containing $x$. By Inequality~(\ref{eq:PBDproof1}), we have
\[\sum_{B\in \mathcal{B}} |B| =\sum_{x\in P} r_x\geq \sum_{x\in P} \frac{n-1}{\tau-1}= \frac{n(n-1)}{\tau-1}.\]
In order to prove \eqref{eq:pbdB}, let $B_0\in\mathcal{B}$ and $|B_0|=\kk$. Define,
\[\tilde{\mathcal{B}}= \{B\setminus B_0 \ : \  B\in\mathcal{B}, B\cap B_0\neq \emptyset\}.\]
We have
\[ \sum_{B\in \tilde{\mathcal{B}} }  |B|=\kk (n-\kk).\]
Now, consider the following set
\[S=\{(x,y) \ : \    x\neq y, x,y \in B, B\in \tilde{\mathcal{B}}\}.\]
We have
\begin{equation}\label{eq:S1}
|S|=\sum_{B\in \tilde{\mathcal{B}} } |B| (|B|-1)\geq \frac{1}{|\tilde{\mathcal{B}}|} \left(\sum_{B\in \tilde{\mathcal{B}} } |B|\right)^2 - \sum_{B\in \tilde{\mathcal{B}} } |B|=\frac{1}{|\tilde{\mathcal{B}}|} \kk^2(n-\kk)^2 - \kk(n-\kk).
\end{equation}

On the other hand,  $S\subseteq \{(x,y) \ : \  x,y\in P\setminus B_0\}$. Thus, 
\begin{equation}\label{eq:S2}
|S|\leq (n-\kk)(n-\kk-1).
\end{equation}
Inequalities (\ref{eq:S1}) and (\ref{eq:S2}) yield
\[|\tilde{\mathcal{B}}|\geq \frac{\kk^2(n-\kk)}{n-1}. \]
Finally, 
\[\sum_{B\in\mathcal{B}} |B|\geq |B_0|+\sum_{B\in \tilde{\mathcal{B}}} (|B|+1)\geq \kk+\kk(n-\kk)+\frac{\kk^2(n-\kk)}{n-1}.\]
Thus, we conclude 
\[\sum_{B\in \mathcal{B}} |B| \geq (n+1) \kk - \frac{\kk^2(\kk-1)}{n-1}\]

To prove Inequality \eqref{eq:pbdC}, let $B_0\in\mathcal{B}$ and $|B_0|=\kk$ and assume that $\mathcal{B}$ has $u$ blocks of size 2 intersecting $B_0$. 
Define, 
\[\hat{\mathcal{B}}= \{B\setminus B_0 \ :\  B\in\mathcal{B}, \ B\cap B_0\neq \emptyset, \ |B|\geq 3\}.\]
Thus,
\begin{equation*}
\binom{n-\kk}{2}\geq \sum_{B\in\hat{\mathcal{B}}} {\binom{|B|}{2}}\geq\sum_{B\in\hat{\mathcal{B}}}(|B|-1).
\end{equation*}
Also,
\begin{equation*}
\kk(n-\kk)= u+\sum_{B\in \hat{\mathcal{B}}} |B|.
\end{equation*}
Hence, 
\begin{align*}
\sum_{B\in \mathcal{B}}|B| &\geq  |B_0|+ 2u+ \sum_{B\in \hat{\mathcal{B}}} (|B|+1) = \kk+ 2\kk(n-\kk)-  \sum_{B\in \hat{\mathcal{B}}} (|B|-1) \\
& \geq  \kk+ 2\kk(n-\kk)- \binom{n-\kk}{2}.
\end{align*}

Now, assume that $\kk\geq n/2$ and $B_0=\{x_1,\ldots, x_k\}$. We provide a PBD with a block $B_0$  for which
equality holds in (\ref{eq:pbdC}). Consider a proper edge coloring of $K_{n-\kk}$ by $n-\kk$ colors and let $C_1,\ldots, C_{n-\kk}$ be color classes. Each $C_i$ is a collection of subsets of size $2$. For every $i$,  $1\leq i\leq n-\kk$, add $x_i$ to each member of $C_i$. Now, we have exactly $(n-\kk)(n-\kk-1)/2$ blocks of size $3$. By adding missing pairs as blocks of size $2$, we get a PBD  $(P,\mathcal{B})$ on $n$ points, with blocks of size $2$ and $3$ and a block of size $\kk$. In fact, each block of size $3$ contains two pairs from the set $\{(x,y) \ :\  x\in B_0, y\not\in B_0\}$. Hence,

\begin{align*}
\sum_{B\in\mathcal{B}}|B| &= \kk+\frac{3(n-\kk)(n-\kk-1)}{2}+ 2 (\kk(n-\kk)- (n-\kk)(n-\kk-1))\\
& =\kk-\frac{(n-\kk)(n-5\kk-1)}{2}.
\end{align*}
\end{proof}
\begin{remark}
Let $(P,\mathcal{B})$ be a PBD with $n$ points where $\tau$ is the maximum size of blocks in $\mathcal{B}$. It is easy to check that among the lower bounds (\ref{eq:pbdA}), (\ref{eq:pbdB}) and (\ref{eq:pbdC}), if  $1\leq \tau\leq  (\sqrt{4n-3}+1)/2$, then (\ref{eq:pbdA}) is the best one, if $ (\sqrt{4n-3}+1)/2\leq \tau\leq (n-1)/2$, then (\ref{eq:pbdB}) is the best one and if $(n-1)/2\leq \tau \leq n-1$, then (\ref{eq:pbdC}) is the best one. The diagram of the lower bounds in terms of $\tau$ are depicted in Figure~\ref{fig:PBD} for $n=21$. 

\begin{center}
\begin{tikzpicture}
\pgfdeclareimage[width=12cm]{main}{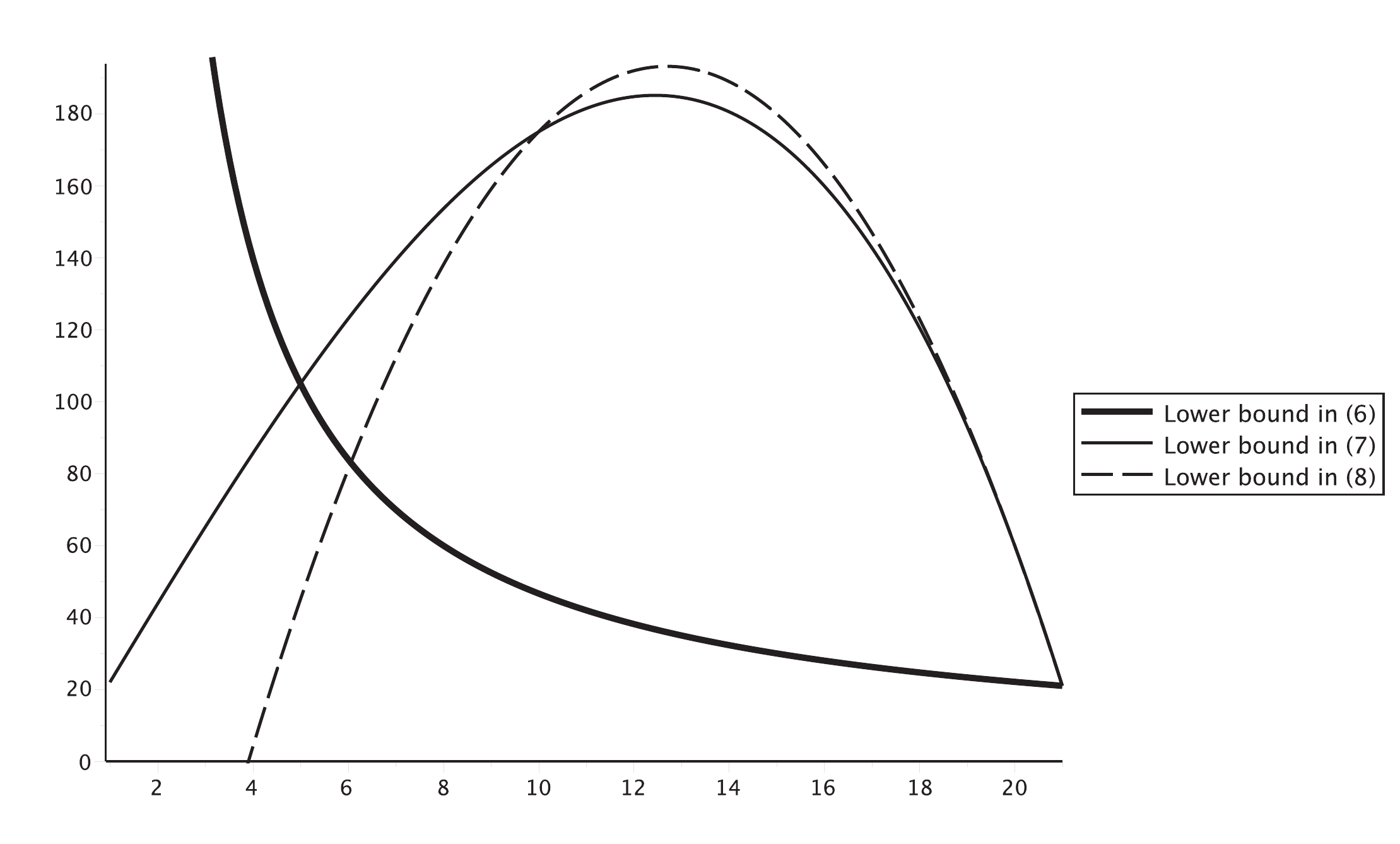}
\pgftext{\pgfuseimage{main}}
\node at (-6.1cm,.2cm)[rotate=90]{\small ${\mathscr{S}}(21,k)$ };
\node at (-.2cm,-3.7cm){\small $k=\tau$ };
\end{tikzpicture}
\captionof{figure}{Diagram of the lower bounds in (\ref{eq:pbdA}), (\ref{eq:pbdB}) and (\ref{eq:pbdC}) for $n=21$.}\label{fig:PBD}
\end{center}

\end{remark}

Now, we apply Theorem \ref{thm:PBD} to improve the bound in (\ref{eq:pbd}), whenever the PBD does not contain large blocks.
\begin{theorem}\label{thm:LargeBlock}
Let $n\geq 10$ and $(P,\mathcal{B})$ be a PBD on $n$ points and assume that $\mathcal{B}$ contains no block of size larger than $n-\frac{1}{2} ( \sqrt{n} +1)$. Then,  we have
\[\sum_{B\in \mathcal{B}} |B|\geq n(\lfloor\sqrt{n}\rfloor+1)-1.\]
Also, the bound is tight in the sense that equality occurs for infinitely many $n$.
\end{theorem}
\begin{proof}
Let $\tau$ be the maximum size of the blocks in $\mathcal{B}$. If $\tau\leq \sqrt{n}$, then by (\ref{eq:pbdA}),
\[\sum_{B\in\mathcal{B}}|B| \geq \frac{n(n-1)}{\tau-1}\geq \frac{n(n-1)}{\sqrt{n}-1}\geq n(\sqrt{n}+1).\]
Now, suppose that $\tau\geq \lfloor\sqrt{n}\rfloor+1$. Then,  $\mathcal{B}$ contains a block of size larger than or equal $\lfloor\sqrt{n}\rfloor+1$. First assume that $\mathcal{B}$ contains a block of size $\kk$, where $\lfloor\sqrt{n}\rfloor+1\leq \kk\leq \frac{n}{2}$. Then, by (\ref{eq:pbdB}),
\[\sum_{B\in\mathcal{B}}|B|\geq (n+1) \kk - \frac{\kk^2(\kk-1)}{n-1}.\]
The right hand side of the above inequality as a function of $\kk$ takes its minimum on the interval $[\lfloor\sqrt{n}\rfloor+1, \frac{n}{2}]$ at $\lfloor\sqrt{n}\rfloor+1$. Thus, 
\begin{align*}
\sum_{B\in\mathcal{B}}|B|&\geq (n+1) (\lfloor\sqrt{n}\rfloor+1) - \frac{(\lfloor\sqrt{n}\rfloor+1)^2\lfloor\sqrt{n}\rfloor}{n-1}\\
&\geq n (\lfloor\sqrt{n}\rfloor+1) + (\lfloor\sqrt{n}\rfloor+1)(1- \frac{(\sqrt{n}+1)\sqrt{n}}{n-1}) \\
& = n (\lfloor\sqrt{n}\rfloor+1) - \frac{\lfloor\sqrt{n}\rfloor+1}{\sqrt{n}-1} \\
& > n (\lfloor\sqrt{n}\rfloor+1) - 2.
\end{align*}
The last inequality is due to the fact that $n\geq 10$.
Finally, assume that $\mathcal{B}$ contains a block of size $\kk$, where $\frac{n}{2}< \kk\leq n-\frac{1}{2} ( \sqrt{n} +1)$. Then, by (\ref{eq:pbdC})
\[\sum_{B\in \mathcal{B}} |B| \geq \kk-\frac{(n-\kk)(n-5\kk-1)}{2}.\]
 Again, the right hand side of the above inequality as a function of $\kk$ takes its minimum on the interval $[\frac{n}{2},n-\frac{1}{2} ( \sqrt{n} +1)]$ at $n-\frac{1}{2} (\sqrt{n}+1)$. Hence,
 \begin{align*}
\sum_{B\in \mathcal{B}} |B| & \geq  n-\frac{1}{2}(\sqrt{n} +1) -\frac{(\sqrt{n}+1)(-4n+\frac{5}{2} (\sqrt{n}+1)-1)}{4}  \\
&=  n(\sqrt{n}+1)+\frac{3n-7}{8}-\frac{3}{2}\sqrt{n}\\
&>  n(\sqrt{n}+1)-2,
 \end{align*}
 where the last inequality is because $ n\geq 10 $. This completes the proof. 

Finally, in order to prove tightness of the bound, let $q$ be a prime power and $(P,\mathcal{B})$ be an affine plane of order $q$. Suppose that $\{B_1,\ldots, B_{q}\}$ is a parallel class. Add a single new point to all the blocks $B_1,\ldots, B_q$. The new PBD has $n=q^2+1$ points, $q^2$ blocks of size $q$ and $q$ blocks of size $q+1$. Hence,  the sum of its block sizes is 
\[q^3+q^2+q=(q^2+1)(q+1)-1= n(\lfloor \sqrt{n}\rfloor+1)-1. \]
\end{proof}
\section{Sigma clique partition of complement of  graphs }\label{sec:Kn-Km}
Given a graph $G$ and its subgraph $H$, the complement of $H$ in $G$ denoted by $G-H$ is obtained from $G$ by removing all edges (but no vertices) of $H$. If $H$ is a graph on $n$ vertices, then $K_n-H$ is called the complement of $H$ and is denoted by $\overline{H}$.

 In this section, applying  the results of Section~\ref{sec:pbd}, we are going to determine the asymptotic  behaviour of the sigma clique partition number of the graph $K_n-K_m$, when $ m $ is a function of $ n $, as well as Cocktail party graph, the complement of path and cycle on $n$ vertices.

The clique partition number of the graph $K_n-K_m$, for $m\leq n$, has been studied by several authors. 
In order to notice the hardness of determining the exact value of $\cp(K_n-K_m)$, note  that if we could show that $\cp(K_{111}-K_{11})\geq 111$, then we could determine whether there exists a projective plane of order $10$~\cite{pullman82}.
Wallis in \cite{WallisAsymp}, proved that $\cp(K_n-G)\sim n$, if $G$ has $o(\sqrt{n})$ vertices. Also, Erd\H{o}s et al. in \cite{erdos88} showed that $\cp(K_n-K_m)\sim m^2$, if $\sqrt{n}< m< n$ and $m=o(n)$. Moreover, if $m=cn$ and $1/2\leq c\leq 1$, then Pullman et al. in \cite{pullmanII} proved that $\cp(K_n-K_m)=1/2 (n-m) (3m-n-1)$.

In the following theorem, we present upper and lower bounds for $ \scp(K_n-K_m) $ and then we improve these bounds in order to determine asymptotic behaviour of $ \scp(K_n-K_m) $.  
\begin{theorem}\label{thm:kn-km}
For every $m,n$, $1\leq m\leq n$, we have 
\begin{equation}\label{eq:kn-km}
mn-\dfrac{m^2(m-1)}{n-1}\leq\scp(K_n-K_m)\leq(2m-1)(n-m)+1.
\end{equation}
\end{theorem}
\begin{proof}
Adding the clique $K_m$ to every clique partition of $K_n-K_m$ forms a PBD on $n$ points. Thus,  the lower bound is obtained from Inequality~(\ref{eq:pbdB}).

For the upper bound, let $V(K_n)=\{x_1,\ldots, x_n\}$ and $V(K_m)=\{x_{n-m+1},\ldots, x_n\}$. Note that the clique $\{x_1,\ldots, x_{n-m+1}\}$ along with $(m-1)(n-m)$ remaining edges form a clique partition of $K_n-K_m$. Hence,  $\scp(K_n-K_m)\leq (n-m+1)+2(m-1)(n-m)$.  
\end{proof}

In the following theorem, for $m\leq\frac{\sqrt{n}}{2}$, we improve
the lower bound in (\ref{eq:kn-km}).
\begin{theorem}\label{thm:kn-km2}
If $m\leq \frac{\sqrt{n}}{2}$, then 
\[(2m-1) n - O(m^2) \leq \scp(K_n-K_m)\leq  (2m-1) n - \Omega(m^2).\]
\end{theorem}
\begin{proof}
The upper bound holds by (\ref{eq:kn-km}).
For the lower bound, consider an arbitrary clique partition of $K_n-K_m$, say $\mathcal{C}$, and add the clique $K_m$ to obtain a PBD $(P,\mathcal{B})$ with $n$ points. Let $\tau$ be the size of maximum block in $\mathcal{B}$. It is clear that $m\leq \tau\leq n-m+1$. We give the lower bound in the following cases. First note that since $m\leq \sqrt{n}/2$, we have $(2m-1)^2\leq n-1$.

If $\tau\leq \frac{n-1}{2m-1}$, then by (\ref{eq:pbdA}), we have
\[\sum_{C\in\mathcal{C}} |C|\geq (2m-1)n-m.\]

If  $\frac{n-1}{2m-1}\leq \tau\leq n/2$, then  $2m-1\leq \tau\leq n/2$, and  by (\ref{eq:pbdB}),
\[\sum_{C\in\mathcal{C}} |C|\geq (n+1) \tau - \frac{\tau^2(\tau-1)}{n-1}-m. \]
The right hand side of this inequality is increasing as a function of $\tau$ within the interval $[2m-1,n/2]$. Hence, 
\[\sum_{C\in\mathcal{C}} |C|\geq (n+1) (2m-1) - \frac{(2m-1)^2(2m-2)}{n-1}-m\geq  (2m-1) n-m. \]
Finally, if  $n/2\leq \tau\leq n-m+1$, then,  by (\ref{eq:pbdC}),
\[\sum_{C\in\mathcal{C}} |C|\geq\tau-\frac{(n-\tau)(n-5\tau-1)}{2}-m.\]
Consider the right hand side of this inequality as a function of $\tau$ within the interval $[n/2,n-m+1]$. It attains its minimum at $\tau=n-m+1$. Hence, 
\[\sum_{C\in\mathcal{C}} |C|\geq n-2m+1-\frac{(m-1)(5m-4n-6)}{2} = (2m-1) n - O(m^2).\]
\end{proof}

The following lemma is a direct application of Theorem~\ref{thm:LargeBlock} that gives a lower bound for $ \scp(K_n-H) $ in terms of $ \scp(H) $. Here, $\omega(G)$ stands for the clique number of graph $G$. 

\begin{lemma}\label{lem:lower bound K_n-H}
Let $ H $ be a graph on $ m $ vertices. If $  \omega(H)\leq n-\frac{1}{2}(\sqrt{n}+1) $ and $\omega(\overline{H})\leq m-\frac{1}{2}(\sqrt{n}+1)$, then $$ \scp(K_n-H)+\scp(H)\geq n(\lfloor\sqrt{n}\rfloor+1)-1. $$
\end{lemma}
\begin{proof}
Assume that $ \cal C $ is an arbitrary clique partition for $ K_n-H $ and $ \tau $ is the size of largest clique in $ \cal C $. Then,  $ \tau\leq n-m+\omega(\overline{H})\leq n-m+m-\frac{1}{2}(\sqrt{n}+1)=n-\frac{1}{2}(\sqrt{n}+1)$. Also, by assumption, $H$ has no clique of size larger than $n-\frac{1}{2}(\sqrt{n}+1)$. Moreover, every clique partition of $H$ along with every clique partition for $K_n-H$ form a PBD. Hence, by Theorem~\ref{thm:LargeBlock}, $ \scp(K_n-H)+\scp(H)\geq n(\lfloor\sqrt{n}\rfloor+1)-1$. 
\end{proof}

We need the following lemma in order to improve the upper bound in \eqref{eq:kn-km} whenever $ \sqrt{n}\leq m\leq n $. The idea is similar to 
\cite{WallisAsymp} that uses a projective plane of appropriate size to give a clique partition for the graph $K_n-K_m$.
\begin{lemma}\label{lem:design}
Let $ H $ be a graph on $ m $ vertices. If there exists a $(v,k,1)-$design, such that $k\geq m$ and $v-k\geq n-m$, then $\scp(K_n-H)\leq n(v-1)/(k-1)+\scp(\overline{H})-m$.
\end{lemma}
\begin{proof}
Let $ (P,\mathcal{B}) $ be a $ (v,k,1)-$design. 
Select a block $ B_1\in \mathcal{B} $ and delete $ k-m $ points from it. Also, delete $ v-k-(n-m) $ points not in $ B_1 $. Now, consider the remaining points as vertices of $ K_n-H $ and each block except $ B_1 $ as a clique in $ K_n-H $. Thus, $ \scp(K_n-H)\leq r(n-m)+(r-1)m+\scp(\overline{H})=nr-m+\scp(\overline{H})$, where $ r=(v-1)/(k-1) $ is the number of blocks containing a single point.
\end{proof}
We are going to apply Lemma~\ref{lem:design} to projective planes and provide a clique covering for $ K_n-H $. Since the existence of projective planes of order $ q $ is only known for prime powers, we need the following well-known theorem to approximate an integer by a prime.  
\begin{thm}{\em\cite{baker}}\label{lem:prime gap}
There exists a constant $ x_0 $ such that for every integer $ x > x_0 $, the interval $ [x , x+x^{.525}] $ contains  prime numbers.
\end{thm}
The following two theorems determine asymptotic behaviour of $ \scp(K_n-K_m) $, when $\sqrt{n}/2\leq m$ and $ m=o(n) $. 
\begin{theorem}\label{thm:kn--km}
Let $ H $ be a graph on $ m $ vertices. If $\frac{\sqrt{n}}{2}\leq m\leq \sqrt{n}$, then $\scp(K_n-H)\leq (1+o(1))\, n\sqrt{n}$. Moreover,  $ \scp(K_n-K_m)= (1+o(1))\, n\sqrt{n}$.
\end{theorem}
\begin{proof}
Let $q$ be the smallest prime power greater than or equal to $\sqrt{n}$. By Theorem~\ref{lem:prime gap}, we have $\sqrt{n}\leq q\leq \sqrt{n}+\sqrt{n}^{.525}$. Thus, 
$ q\geq\sqrt{n}>m-1 $
and
$ q^2\geq n\geq n-m $. Since there exists a projective plane of order $q$, by Lemma~\ref{lem:design}, we have
\[\scp(K_n-H)\leq n (q+1)-m+\scp(\overline{H}) \leq n (q+1)-m+\frac{m^2}{2},\]
where the last inequality is due to the fact that sigma clique partition number of every graph on $ n $ vertices is at most $ n^2/2 $ \cite{chung81,gyori80}. Hence,
\[\scp(K_n-H)\leq n^{1.5}+n^{1.2625}+1.5\ n= (1+o(1))\, n\sqrt{n}.\]
Also, by Lemma \ref{lem:lower bound K_n-H}, $ \scp(K_n-K_m)\geq (1+o(1))\, n\sqrt{n} $.
\end{proof}
In the following theorem, for $\sqrt{n}\leq m\leq n$, we improve
the upper bound in (\ref{eq:kn-km}).
\begin{theorem}\label{thm:kn-km1}
If $\sqrt{n}\leq m\leq n$, then $\scp(K_n-K_m)\leq (1+o(1))\, nm$ and if in addition $ m=o(n) $, then $\scp(K_n-K_m)= (1+o(1))\,nm$.
\end{theorem}
\begin{proof}
Let $\sqrt{n}\leq m\leq n$, and also let $q$ be the smallest prime power which is greater than or equal to $m$. By Lemma \ref{lem:prime gap}, $m\leq q\leq m+m^{.525}$. Thus,  $q=(1+o(1))\, m$. Since there exists a projective plane of order $q$, by Lemma~\ref{lem:design}, we have
\[\scp(K_n-K_m)\leq n (q+1)-m = (1+o(1))\, nm.\]
On the other hand, when $ m=o(n) $, Inequality~(\ref{eq:kn-km}) yields  $ \scp(K_n-K_m)\geq (1+o(1))\, nm$, which completes the proof.
\end{proof}
Theorems \ref{thm:kn-km2}, \ref{thm:kn--km} and \ref{thm:kn-km1}  make clear asymptotic behaviour of $K_n-K_m$ in case $m=o(n)$.
\begin{cor}\label{cor: K_n-K_m}
Let $m$ be a function of $n$. Then
\begin{itemize}
\item[\rm i)]  If $m\leq \frac{\sqrt{n}}{2}$, then $\scp(K_n-K_m)\sim (2m-1)n$.
\item[\rm ii)] If $\frac{\sqrt{n}}{2}\leq m\leq \sqrt{n}$, then $\scp(K_n-K_m)\sim n\sqrt{n}$.
\item[\rm iii)] If $m\geq \sqrt{n}$ and $m=o(n)$, then $\scp(K_n-K_m)\sim mn$.
\end{itemize}
\end{cor}
In what follows, we consider the case $m=cn$, where $c$ is a constant.
 First note that if $1/2\leq c\leq 1$, then by Theorem~\ref{thm:PBD}, since $m\geq n/2$, there exists a PBD on $n$ points with a block of size $m$, for which equality holds in (\ref{eq:pbdC}). Hence, we have $\scp(K_n-K_m)= \frac{(1-c)}{2}\bigg((5c-1)n^2+n\bigg)$.
 In order to deal with the case $c<1/2$, we need the following well-known existence theorem of resolvable designs.
 \begin{thm}{\em \cite{Lu}}\label{thm:resol}
 Given any integer $k\geq 2$, there exists an integer $v_0(k)$ such that for every $ v\geq v_0(k) $, a $(v,k,1)-$resolvable design exists if and only if $v \overset{k}{\equiv} 0$ and $v-1\overset{k-1}{\equiv} 0$.
 \end{thm}
\begin{theorem}
Let  $0<c<1/2$ be a constant and $m,n$ be some integers satisfying $m=cn$. Then
\begin{equation}\label{eq:cn}
c(1-c^2)n^2+ \Omega(n)\leq \scp(K_n-K_m) \leq  \frac{(1-c)(\lfloor 1/c\rfloor-c)}{\lfloor 1/c\rfloor(\lfloor 1/c\rfloor-1)} n^2+O(n).
\end{equation}
In particular, if $1/c$ is integer, then $\scp(K_n-K_m) \sim c(1-c^2)n^2$.
\end{theorem}
\begin{proof}
The lower bound in (\ref{eq:cn}) is obtained from the lower bound in (\ref{eq:kn-km}). For the upper bound, let $k=\lfloor 1/c\rfloor$ and define $v$ as the smallest number greater than or equal to $n-m$ which satisfies the conditions of Theorem~\ref{thm:resol}. Without loss of generality we can assume that $n$ is sufficiently large, i.e. $ n\geq v_0(k) $. Thus, we have $v\leq n-m+k^2$ and by Theorem~\ref{thm:resol}, there exists a $(v,k,1)-$resolvable design. Remove $v-n+m$ points from such a design to obtain a PBD $ (P,\mathcal{B}) $ on $n-m$ points whose blocks are partitioned into  $t=(v-1)/(k-1)$ parallel classes. First, we show that $m\leq t$. Note that
\[m-t=cn-\frac{v-1}{k-1}\leq cn-\frac{(1-c)n-1}{k-1}=\frac{(ck-1)n+1}{k-1}.\]
If $k=2$, then $ck<1$ and $m-t<1$. Also, if $k>2$, then $ck\leq 1$ and thus $m-t\leq 1/(k-1)<1$. Therefore, $m\leq t$.

Now, let $v_1,\ldots, v_m$ be $m$ new points and for every $i$, $1\leq i\leq m$, add point $v_i$ to all blocks of $i$-th parallel class. These blocks form a clique partition $\mathcal{C}$ for $K_n-K_m$, where
\begin{equation*}
\sum_{C\in\mathcal{C}}|C|\leq \sum_{B\in\mathcal{B}}|B|+\frac{v}{k} m= (n-m)\frac{v-1}{k-1}+\frac{mv}{k}.
\end{equation*}
Hence,
\begin{align*}
\sum_{C\in\mathcal{C}}|C|&\leq \left(\frac{(1-c)^2}{k-1}+ \frac{c(1-c)}{k}\right)n^2+O(n)\\
&= \frac{(1-c)(k-c)}{k(k-1)} n^2+O(n).
\end{align*}
\end{proof}  
%
We close the paper by proving that if $ G $ is Cocktail party graph, complement of path or cycle on $ n $ vertices, then $ \scp(G)\sim n\sqrt{n} $.  
Given an  even positive integer $n$, Cocktail party graph $T_n$ is obtained from the complete graph $K_{n}$ by removing a perfect matching. If $n$ is an odd positive integer, then $T_n$ is obtained from $T_{n+1}$ by removing a single vertex. 
In \cite{Wallis87,gregory86} it is proved that if $G$ is Cocktail party graph or complement of a path or a cycle on $n$ vertices, then $n\leq \cp(G)\leq (1+o(1))\, n \log\log n$ and it is conjectured that for such a graph, $\cp(G)\sim n$. 

\begin{theorem}\label{thm:complment of path}
Let $ P_n $ be the path on $ n $ vertices. Then,  $ \scp(\overline{P_n})\sim n^{3/2} $.
\end{theorem}
\begin{proof}
By Lemma~\ref{lem:lower bound K_n-H}, we have $ \scp(\overline{P_n})\geq n^{3/2}-2n-3$. Now, by induction on $n$, we prove that there exists a constant $c$, such that $\scp(\overline{P_n})\leq n^{3/2}+c\, n^{13/10}$.  The idea is similar to \cite{Wallis87}.\\
Let $ d=\lfloor\sqrt{n}\rfloor $, $ e=\lceil\frac{n}{d}\rceil $ and $ q $ be the smallest prime greater than $ \sqrt{n} $. By Lemma~\ref{lem:prime gap}, $q\leq \sqrt{n}+n^{3/10}$. In an affine plane of order $ q $, choose a parallel class, say $ C_1 $, and delete $ q-d $ blocks in $C_1$. Then,  remove $ q-e $ blocks in a second parallel class, say $ C_2 $. The collection of remaining blocks is a PBD on $de$ points.

Assume that $ a_{ij} $ is the intersection point of block $ i $ of $ C_1 $ and block $ j $ of $ C_2 $ in the remaining PBD. Thus, 
 $C_1=\{\{a_{i1}, a_{i2},\ldots , a_{ie}\} \ : \  1\leq i\leq d\}$ and $C_2=\{\{a_{1j}, a_{2j},\ldots , a_{dj}\} \ : \  1\leq j\leq e \}$.
Now, replace each block in $ C_2 $ by members of a clique partition of a copy of $\overline{P_d }$ on the same vertices. Also, replace each of the blocks $ \{a_{11}, a_{12},\ldots , a_{1e}\} $ and $ \{a_{d1}, a_{d2},\ldots , a_{de}\} $ in $ C_1 $ by members of a clique partition of a copy of $\overline{P_e}$ on the same vertices. In fact, we have replaced $ e+2 $ blocks by some clique partitions of complement of paths and $ q(q+1)-(e+2) $ blocks are left unchanged. It can be seen that the resulting collection, is a partition of all edges of $\overline{P_{de}}$ except $(e-1)$ edges namely $a_{11}a_{12}, a_{d2}a_{d3}, a_{13}a_{14}, a_{d4}a_{d5},\dots$ . Adding these $e-1 $ edges to this collection comprise a clique partition for $\overline{P_{de}}$. Hence,

\begin{align*}
\scp(\overline{P_n})\leq\scp(\overline{P_{de}})&\leq qde -2e+e\scp(\overline{P_d})+2\scp(\overline{P_e})+2(e-1).
\end{align*}
Since $ e\leq d+3 $, $ \scp(\overline{P_e})\leq\scp(\overline{P_d})+6d $. Thus,
\begin{align*}
\scp(\overline{P_n})&\leq qd(d+3)+(d+5)\scp(\overline{P_d})+12d.
\end{align*}
Therefore, by the induction hypothesis, we have
\begin{align*}
\scp(\overline{P_n})&\leq (\sqrt{n}+n^{3/10}) \sqrt{n}(\sqrt{n}+3)+(\sqrt{n}+5)(n^{3/4}+c\, n^{13/20})+12\sqrt{n}\\
&\leq n^{3/2}+ (1+o(1))\, n^{13/10}\\
& \leq n^{3/2}+c\, n^{13/10}.
\end{align*}
\end{proof}
Asymptotic behavior of $\scp(T_n)$ and $\scp(\overline{C_n})$ can be easily determined using $\scp(\overline{P_n})$, as follows.
\begin{cor}
Let $ T_n $ and $ C_n $ be Cocktail party graph and cycle on $ n $ vertices, respectively. Then,  $ \scp(\overline{C_n})\sim n^{3/2} $ and $ \scp(T_n)\sim n^{3/2} $.
\end{cor}
\begin{proof}
By Lemma \ref{lem:lower bound K_n-H}, $ \scp(\overline{C_n})\geq n^{3/2}-2n-1 $ and $ \scp(T_n)\geq n^{3/2}-n-1 $.\\
Note that $\overline{P_n}$ is obtained from $\overline{C_{n+1}}$ by removing an arbitrary vertex $v$. Adding $n-2$ edges incident with $v$ to any clique partition of $ \overline{P_{n}} $ forms a clique partition for $ \overline{C_{n+1}} $. Therefore, $ \scp(\overline{C_{n+1}})\leq\scp(\overline{P_n})+2(n-1) $. Also, adding at most $n/2$ edges to any clique partition for $\overline{P_n}$ forms a clique partition for $T_n$. Thus, $ \scp(T_n)\leq\scp(\overline{P_n})+2\frac{n}{2} $. Hence, by Theorem~\ref{thm:complment of path}, $ \scp(\overline{C_n}), \scp(T_n)\leq (1+o(1))\, n^{3/2}$. 
\end{proof}

\end{document}